\documentclass[12pt]{article}
\usepackage{fullpage}
\usepackage{mathrsfs,pstricks,ifpdf,tikz}
\usepackage{amsfonts}
\usepackage{enumerate}

\usepackage{makecell}

\usepackage{latexsym}
\usepackage{amsmath,amssymb}
\usepackage{amsthm}
\usepackage{ifthen,calc}
\usepackage{color}
\usepackage{graphicx}
\usepackage{latexsym}
\usepackage{multirow}
\usepackage{diagbox}
\usepackage{float}
\usepackage[format=hang, margin=10pt]{caption}

\newtheorem{theorem}{Theorem}[section]

\newtheorem{conjecture}[theorem]{Conjecture}
\newtheorem{remark}[theorem]{Remark}
\newtheorem{lemma}[theorem]{Lemma}

\theoremstyle{definition}

\newtheorem{example}[theorem]{Example}

\newtheorem{construction}[theorem]{Construction}

\begin{document}

\title{Counterexamples to two conjectures on mean color numbers of graphs}

\author{
Wushuang Zhai\thanks{School of Mathematics, Tianjin University, Tianjin, China: 18706628112@163.com.}\,
\qquad
Yan Yang\thanks{School of Mathematics and KL-AAGDM, Tianjin University, Tianjin, China: yanyang@tju.edu.cn.
Supported by National Natural Science Foundation of China under Grant 12371350. Corresponding author.}\,
}

\date{\today}

\maketitle

\begin{abstract}
The mean color number of an $n$-vertex graph $G$, denoted by $\mu(G)$, is the average number of colors used in all proper $n$-colorings of $G$.
For any graph $G$ and a vertex $w$ in $G$, Dong (2003) conjectured that if $H$ is a graph obtained from a graph $G$ by deleting all but one of the edges which are incident to $w$, then $\mu(G)\geq \mu(H)$;  and also conjectured that $\mu(G)\geq \mu((G-w)\cup K_1)$. We prove that there is an infinite family
of counterexamples to these two conjectures.

\medskip



\end{abstract}

\section{Introduction}

Let $G$ be a simple graph, $V(G)$ and $E(G)$ be its vertex set and edge set respectively.  For any $u\in V(G)$, let $N_G(u)$ denote the set of neighbors of $u$ in $G$ and $d_G(u)$ denote the degree of $u$ in $G$, and $G-u$ denote the subgraph of $G$ obtained by removing $u$ together with the edges incident with it from $G$.
Denote by $G+v$ the graph $G$ with a new vertex $v$ added and with no edges connected to $v$. For $u,v\in V(G)$, we denote $G/uv$ the graph obtained from $G$ by identifying $u$ and $v$ and replacing multiedges by single ones.
If $uv\in E(G)$, we denote by $G-uv$ the graph obtained by deleting edge $uv$ from $G$.
The complete graph on $n$ vertices is denoted by $K_n.$ The\textit{ union} $G\cup H$ of two graph $G$ and $H$ is the graph with vertex set $V(G)\cup V(H)$,
and edge set  $E(G)\cup E(H).$ The \textit{join} $G\vee H$ of two vertex disjoint graphs $G$ and $H$ is obtained from $G\cup H$ by joining every vertex of $G$ to every vertex of $H$.

For a positive integer $\lambda$,  a {\it proper $\lambda$-coloring} of $G$ is a mapping $c: V(G)\rightarrow \{1,\ldots,\lambda\}$ such that
$c(u)\neq c(v)$ whenever $uv\in E(G)$. Two $\lambda$-colorings $c$ and $c'$ are distinct if $c(v)\neq c'(v)$ for some vertex $v$ of $G$. Let $P(G,\lambda)$ be the number of distinct proper $\lambda$-colorings of $G$, it is a polynomial in $\lambda$ and called {\it chromatic polynomial} of $G$. It was introduced by Birkhoff \cite{B1912} in 1912 with the hope of proving the Four Color Conjecture.

We denote $(\lambda)_{k}=\lambda(\lambda-1)\cdots(\lambda-k+1)$ the $k$th {\it falling factorial} of $\lambda$, and denote $\alpha(G,k)$ the number of partitions of
$V(G)$ into exactly $k$ nonempty independent sets, then $\alpha(G,k)(\lambda)_{k}$ is the number of $\lambda$-colorings of $G$ in which exactly $k$ colors are used.
It is well known that for the chromatic polynomial of an $n$-vertex graph $G$,
\begin{equation}P(G,\lambda)=\sum\limits_{k=1}^{n}\alpha(G,k)(\lambda)_{k}.\end{equation}
For any $n$-vertex graph $G$, there exist $n$-colorings of $G$, the {\it  mean color number} $\mu(G)$ of $G$, defined by Bartels and Welsh \cite{B1995},
is the average of number of colors used in all $n$-colorings of $G$. From the definition,
\begin{equation}\mu(G)=\frac{\sum\limits_{k=1}^{n}k\alpha(G,k)(n)_{k}}{\sum\limits_{k=1}^{n}\alpha(G,k)(n)_{k}}.\end{equation}

By applying $(1)$, Bartels and Welsh \cite{B1995} gave a formula to compute the mean color number of graphs.
\begin{theorem}[\cite{B1995}]\label{1} If $G$ is an $n$-vertex graph, then
$$\mu(G)=n\Big(1-\frac{P(G,n-1)}{P(G,n)}\Big).$$\end{theorem}
In  \cite{B1995}, Bartels and Welsh conjectured that if $H$ is a spanning subgraph of $G$, then $\mu(G)\geq \mu(H)$. But Mosca \cite{M1998} found counterexamples. Dong
\cite{D2003,D2005} proved that this conjecture holds under some conditions.  In \cite{B1995}, Bartels and Welsh also conjectured that for any $n$-vertex graph $G$,
$\mu(G)\geq \mu(O_n)$, where $O_n$ is the empty graph with $n$ vertices. Dong proved this conjecture in \cite{D2000}. Furthermore, Dong \cite{D2005} proved that for any
$n$-vertex graph $G$, $\mu(G)\geq \mu(Q)$, where $Q$ is any $2$-tree with $n$ vertices and $G$ is any graph whose vertex set has an ordering $x_1, x_2,\ldots, x_n$ such that
$x_i$ is contained in a $K_3$ of $G[V_i]$ for $i=3,4,\ldots,n$, where $V_i=\{x_1,x_2,\ldots, x_i\}.$

In 2003, Dong \cite{D2003} posed the following two conjectures.

\begin{conjecture}[\cite{D2003}]\label{c2} For any graphs $G$ and a vertex $w$ in $G$, $\mu(G)\geq \mu((G-w)\cup K_1)$.
\end{conjecture}

\begin{conjecture}[\cite{D2003}]\label{c1} For any graph $G$ and a vertex $w$ in $G$ with $d(w)\geq 1$, if $H$ is a graph obtained from $G$ by deleting all  but one of the edges which are incident to $w$, then $\mu(G)\geq\mu(H)$.
\end{conjecture}

We note that for any graph $G$ and a vertex $w$ in $G$, if $d(w)=0$, then $(G-w)\cup K_1\cong G$, Conjecture \ref{c2} holds;
if $d(w)=1$, for the graph $H$ in Conjecture \ref{c1}, we have $H\cong G$, Conjecture \ref{c1} holds.
Dong \cite{D2003}, Long and Ren \cite{L2022} proved these two conjectures hold for some kinds of graphs.
We will give counterexamples to Conjecture \ref{c2} and Conjecture \ref{c1} when
$d(w)\geq 1$ and $d(w)\geq 2$ respectively.

\section{Counterexamples to Conjectures \ref{c2} and \ref{c1}}

For graphs $G_1$ and $G_2$, Dong \cite{D2003} defined \begin{equation}\tau(G_1,G_2,\lambda)=P(G_1,\lambda)P(G_2,\lambda-1)-P(G_1,\lambda-1)P(G_2,\lambda)\end{equation}
to compare $\mu(G_1)$ with $\mu(G_2)$. From Theorem \ref{1}, one can deduce the following result.

\begin{lemma}[\cite{D2003}]\label{3} For any $n$-vertex graphs $G_1$ and  $G_2$, the inequality $\mu(G_1)< \mu(G_2)$ is equivalent to $\tau(G_1, G_2,n)<0$.
\end{lemma}

For more than two graphs, Dong proved the following Lemma.
\begin{lemma}[\cite{D2003}]\label{6} For any $n$-vertex graphs $G_1$, $G_2$ and $G_3$, if $\tau(G_1,G_2,n)<0$ and $\tau(G_2,G_3,n)<0$, then $\tau(G_1,G_3,n)<0$.
\end{lemma}

Now we give the counterexamples to Conjecture \ref{c2}.
\begin{construction}\label{51} Let $G_0$ be a graph which contains a subgraph $K_k$ $(k\geq 2)$ and $|V(G_0)|=k+s-2$ $(s\geq 2)$. Let $G_1$ be a graph obtained from $G_0$ by adding two new vertices $u$ and $v$, joining $u$ to $i+t$ $(i\geq 1, t\geq 0)$ vertices in $K_k$,  joining $v$ to $j$ $(j\geq 1)$ vertices in $K_k$,  joining $u$ to $v$,  and satisfying $|N_{G_1}(u)\cup N_{G_1}(v)|=k+2$, $|N_{G_1}(u)\cap N_{G_1}(v)|=t$.
\end{construction}

\begin{theorem}\label{5} Let $G_0$ and $G_1$ be graphs as constructed in the Construction \ref{51}, and $G_2=(G_1-v)\cup K_1$. If $t<j$ and $i> \frac{j^{3}+2(s-t-1)j^2+(t^2+2t-2ts+s^2-s)j+s^2-s}{j-t}$, then $\mu(G_1)<\mu(G_2)$.
\end{theorem}

\begin{proof}  Let $k+s=n$. Both $G_1$ and $G_2$ are $n$-vertex graphs. From the construction of $G_1$, we have $N_{G_1}(u)\cup N_{G_1}(v)=V(K_k)\cup \{u,v\}$
and $i+j=k$. Firstly, we compute the chromatic polynomials $P(G_1,\lambda)$ and $P(G_2,\lambda)$.
In the graph $G_1-uv$, the degrees of $u$ and $v$ are $i+t$ and $j$ respectively,
so $$P(G_1-uv,\lambda)=P(G_0,\lambda)(\lambda-i-t)(\lambda-j).$$
Because $N_{G_1-uv}(u)\cup N_{G_1-uv}(v)=V(K_k)$, we have $$P(G_1/uv,\lambda)=P(G_0,\lambda)(\lambda-k).$$
By using the edge deletion-contraction formula, we have
\begin{eqnarray}P(G_1,\lambda)&=&P(G_1-uv,\lambda)-P(G_1/uv,\lambda)\nonumber\\
&=&P(G_0,\lambda)(\lambda-i-t)(\lambda-j)-P(G_0,\lambda)(\lambda-k).
\end{eqnarray}
One can also compute that \begin{equation}P(G_2,\lambda)=P(G_0,\lambda)\lambda(\lambda-i-t).\end{equation}
Combining $(3)$, $(4)$, $(5)$ and $i+j=k$, $k+s=n$,  we have \begin{eqnarray*}&&\tau(G_1,G_2,n)\\&=&P(G_1,n)P(G_2,n-1)-P(G_1,n-1)P(G_2,n)\\
&=&P(G_0,n)\big((n-i-t)(n-j)-(n-k)\big)P(G_0,n-1)(n-1)(n-1-i-t)\\
&&-P(G_0,n-1)\big((n-1-i-t)(n-1-j)-(n-1-k)\big)P(G_0,n)n(n-i-t)\\
&=&P(G_0,n)P(G_0,n-1)\big(j^3+2(s-t-1)j^2+(t^2+2t-2ts+s^2-s)j+s^2-s+(t-j)i\big).
\end{eqnarray*}
Because $G_0$ has $n-2$ vertices, both $P(G_0,n)$ and $P(G_0,n-1)$ are more than 0.
If $i> \frac{j^{3}+2(s-t-1)j^2+(t^2+2t-2ts+s^2-s)j+s^2-s}{j-t}$, then $j^3+2(s-t-1)j^2+(t^2+2t-2ts+s^2-s)j+s^2-s+(t-j)i<0$, $\tau(G_1,G_2,n)<0$. From Lemma \ref{3}, $\mu(G_1)<\mu(G_2).$
\end{proof}


\begin{lemma} \label{7} Let $G_0$ be a graph with $n-1$ vertices, $G=G_0+w$ 
and $H$ be the graph obtained from $G$ by joining $w$ to a vertex in $G_0$, then $\tau(G,H,n)<0$.
\end{lemma}

\begin{proof} It is easy to compute that $P(G,\lambda)=\lambda P(G_0,\lambda)$ and $P(H,\lambda)=(\lambda-1) P(G_0,\lambda)$. So we have
\begin{eqnarray*}\tau(G,H,n)&=&P(G,n)P(H,n-1)-P(G,n-1)P(H,n)\\
&=&nP(G_0,n)(n-2)P(G_0,n-1)-(n-1)P(G_0,n-1)(n-1)P(G_0,n)\\
&=&-P(G_0,n)P(G_0,n-1)\\
&<&0.
\end{eqnarray*}
\end{proof}



\begin{theorem} \label{8} 
A counterexample to Conjecture \ref{c2} is also a counterexample to Conjecture \ref{c1}.
\end{theorem}

\begin{proof} Let $G_1$ be a counterexample to Conjecture \ref{c2}, i.e., there exists a vertex $v\in V(G_1)$, and $G_2=(G_1-v)\cup K_1$, such that $\mu(G_1)<\mu(G_2)$. Suppose that $|V(G_1)|=n$, then  $\tau(G_1,G_2,n)<0$. Let $G_3$ be the graph obtained from $G_1$ by deleting all but one of the edges which are incident to $v$. From Lemma \ref{7}, we have $\tau(G_2,G_3,n)<0$.
From Lemma \ref{6}, $\tau(G_1,G_3,n)<0$ follows. So we have $\mu(G_1)<\mu(G_3)$, $G_1$ is also a counterexample to Conjecture \ref{c1}.
\end{proof}

\begin{example} For $G_1$ in Construction \ref{51}, when $t=0$ and $s=2$, $G_1\cong (K_i+v)\vee(K_{j}+u)$. Let $G_2=(G_1-v)\cup K_1$ and
$G_3$ be the graph obtained from $G_1$ by deleting all but one of the edges which are incident to $v$. From Theorems \ref{5} and \ref{8},
if $i,j\geq 1$ and $i>(j+1)^2+1+\frac{2}{j}$,  then $\mu(G_1)<\mu(G_2)$ and $\mu(G_1)<\mu(G_3)$.
\end{example}

\begin{remark} Let $G_1$ be the graph defined in Construction \ref{51}, and $k+s=n$. It is easy to check that
\begin{eqnarray*}&&\tau(G_1,G_1-uv,n)\\&=&P(G_1,n)P(G_1-uv,n-1)-P(G_1,n-1)P(G_1-uv,n)\\
&=&P(G_0,n)\big((n-i-t)(n-j)-(n-k)\big)P(G_0,n-1)(n-j-1)(n-1-i-t)\\
&&-P(G_0,n-1)\big((n-1-i-t)(n-1-j)-(n-1-k)\big)P(G_0,n)(n-j)(n-i-t)\\
&=&P(G_0,n)P(G_0,n-1)\big((n-j)(n-k-1)(n-i-t)-(n-j-1)(n-k)(n-i-t-1)\big)\\
&=&P(G_0,n)P(G_0,n-1)\big(i(t-j)+s(s-1)\big).
\end{eqnarray*}
If $i>\frac{s(s-1)}{j-t}$, then $i(t-j)+s(s-1)<0$, $\tau(G_1,G_1-uv,n)<0$. From Lemma \ref{3}, $\mu(G_1)<\mu(G_1-uv).$
Thus, the graph $G_1$ is a counterexample to  Bartels and Welsh's \cite{B1995} conjecture   which state that if $H$ is a spanning subgraph of $G$, then $\mu(G)\geq \mu(H)$.
The former counterexamples given by Mosca\cite{M1998} are graphs whose chromatic numbers are even or $n-2$ ($n$ is the number of vertices).
In our counterexamples, the chromatic number of $G_1$ can be any $k$ $(k\geq 4)$, which is a supplement to Mosca's counterexamples.
\end{remark}

\bibliography{bibfile}
\end{document}